\documentclass{amsart}
\usepackage{amsmath}
\usepackage{amsfonts}

\newtheorem{theorem}{Theorem}
\theoremstyle{plain}

\newtheorem{remark}{Remark}

\numberwithin{equation}{section}

\begin{document}
\title[More on Jensen functional and convexity ]{More on Jensen functional
and convexity }
\author{Shoshana Abramovich}
\address{Department of Mathematics, University of Haifa, Haifa, Israel}
\email{abramos@math.haifa.ac.il }
\date{May 22, 2024}

\begin{abstract}
In this paper we prove\ results on the difference between a normalized
Jensen functional and the sum of other normalized Jensen functionals for
convex function.
\end{abstract}

\maketitle

\section{\protect\bigskip Introduction}

In this paper we prove\ results on the difference between a normalized
Jensen functional and the sum of other normalized Jensen functionals for
convex function.

Jensen functional is:

\begin{equation*}
J_{n}\left( f,\mathbf{x},\mathbf{p}\right) =\sum_{i=1}^{n}p_{i}f\left(
x_{i}\right) -f\left( \sum_{i=1}^{n}p_{i}x_{i}\right) .
\end{equation*}

We start with the theorem of S. S. Dragomir:

\begin{theorem}
\label{Th1} \cite{D} \textit{Consider the normalized Jensen functional where 
}$f:C\longrightarrow 
\mathbb{R}
$\textit{\ is a convex function on the convex set }$C$ in a real linear
space,\textit{\ }$\mathbf{x}=\left( x_{1},...,x_{n}\right) \in C^{n},$ and%
\textit{\ \ }$\mathbf{p}=\left( p_{1},...,p_{n}\right) ,$\textit{\ \ }$%
\mathbf{q}=\left( q_{1},...,q_{n}\right) $\textit{\ are non-negative
n-tuples satisfying }$\sum_{i=1}^{n}p_{i}=1,$\textit{\ \ }$%
\sum_{i=1}^{n}q_{i}=1,$\textit{\ \ }$p_{i},$\textit{\ }$q_{i}>0,$\textit{\ \ 
}$i=1,...,n$\textit{. Then } 
\begin{equation*}
MJ_{n}\left( f,\mathbf{x},\mathbf{q}\right) \geq J_{n}\left( f,\mathbf{x},%
\mathbf{p}\right) \geq mJ_{n}\left( f,\mathbf{x},\mathbf{q}\right) ,\qquad
\end{equation*}%
provided that 
\begin{equation*}
m=\min_{1\leq i\leq n}\left( \frac{p_{i}}{q_{i}}\right) ,\quad M=\max_{1\leq
i\leq n}\left( \frac{p_{i}}{q_{i}}\right) .
\end{equation*}
\end{theorem}

This important theorem was extended in \cite{A1}, \cite{S}, \cite{SDB} and 
\cite{TTD} using a similar technique as we use here.

\bigskip

In Theorem \ref{Th2} the right handside of the inequality in Theorem \ref%
{Th1} is extended.

In Theorem \ref{Th2} the following conditions and notations are used:

$0\leq p_{i,1}\leq 1$, $0<q_{i}\leq 1,\
\sum_{i=1}^{n}p_{i,1}=\sum_{i=1}^{n}q_{i}=1$.

Denote $m_{1}=\underset{i=1,...,n}{\min }\left( \frac{p_{i,1}}{q_{i}}\right) 
$, and $s_{1}$ the number of $i$-th for which $m_{1}$ occur.

Define recoursively%
\begin{eqnarray*}
p_{i,k} &=&\left\{ 
\begin{array}{cc}
p_{i,k-1}-m_{k-1}q_{i}, & m_{k-1}\neq \frac{p_{i,k-1}}{q_{i,}} \\ 
\frac{1}{s_{k-1}}m_{k-1}, & m_{k-1}=\frac{p_{i,k-1}}{q_{i}}%
\end{array}%
\right. , \\
m_{k-1} &=&%
\begin{array}{cc}
\underset{1\leq i\leq n}{\min }\left( \frac{p_{i,k-1}}{q_{i}}\right) , & 
k=2,...,N%
\end{array}%
\qquad \qquad
\end{eqnarray*}

\bigskip\ and denote $s_{k-1}$ as the number of cases for which $m_{k-1}$
occurs.

Let also $\mathbf{x}_{1}=\left( x_{1,1},x_{2,1},...,x_{n,1}\right) \in
\left( a,b\right) ^{n}$, and define recoursively 
\begin{equation*}
x_{i,k}=\left\{ 
\begin{array}{cc}
x_{i,k-1\hspace{0in}}, & m_{k-1}\neq \frac{p_{i,k-1}}{q_{i}} \\ 
\sum_{i=1}^{n}q_{i}x_{i,k-1}, & m_{k-1}=\frac{p_{i,k-1}}{q_{i}}%
\end{array}%
\right. ,
\end{equation*}%
$i=1,...n$, $k=2,...,N$.

\begin{theorem}
\label{Th2} \cite[Theorem 5]{A1} Suppose that $f:\left[ a,b\right)
\rightarrow 
\mathbb{R}
,$ $a<b\leq \infty $ is a convex function. Then, for every integer $N,$ 
\begin{equation}
J_{n}\left( f,\mathbf{x}_{1},\mathbf{p}_{1}\right)
-\sum_{k=1}^{N}m_{k}J_{n}\left( f,\mathbf{x}_{k},\mathbf{q}\right) \geq 0,
\label{1.1}
\end{equation}%
where $m_{k}$ and $\mathbf{x}_{k}=$\ $\left( x_{1,k},...,x_{n,k}\right) ,$ $%
k=1,...,N$ are as defined above, and $\mathbf{p}_{1}=\left(
p_{1,1},...,p_{n,1}\right) ,$\ $\mathbf{q=}\left( q_{1},...,q_{n}\right) 
\mathbf{,}$ and $\sum_{i=1}^{n}p_{i,1}=\sum_{i=1}^{n}q_{i}=1,$ and $%
p_{i,1}\geq 0,$\ $q_{i}>0,$ $\ \ i=1,...,n$, $m_{1}=\underset{1\leq i\leq n}{%
\min }\left( \frac{p_{i,1}}{q_{i}}\right) .$
\end{theorem}

\bigskip In Theorem \ref{Th3} the left hand-side inequality in Theorem \ref%
{Th1} is extended.

Denote 
\begin{eqnarray*}
\mathbf{p}_{1} &=&\left( p_{1,1},...,p_{1,n}\right) ,\quad \mathbf{q=}\left(
q_{1},...,q_{n}\right) , \\
\quad \mathbf{x}_{k} &=&\left( x_{1,k},...,x_{n,k}\right) ,\quad k=1,...,N \\
p_{i,1} &\geq &0,\quad q_{i}>0,\quad i=1,...,n,\quad
\sum_{i=1}^{n}p_{i,1}=\sum_{i=1}^{n}q_{i}=1, \\
M_{1} &=&\underset{1\leq i\leq n}{Max}\left( \frac{p_{i,1}}{q_{i}}\right) =%
\frac{p_{j_{1},1}}{q_{j_{1}}},
\end{eqnarray*}%
where $j_{1}$ is a fixed specific integer for which $M_{1}$\ holds.

Define recoursively 
\begin{eqnarray*}
p_{i,k} &=&p_{i,k-1}-M_{k-1}q_{i},\ x_{i,k}=x_{i,k-1},\ \text{when}\
M_{k-1}\neq \frac{p_{i,k-1}}{q_{i}},\ k=2,...,N \\
p_{i,k} &=&p_{i,k-1}-M_{k-1}q_{i},\ x_{i,k}=x_{i,k-1},\ \text{when}\ M_{k-1}=%
\frac{p_{i,k-1}}{q_{i}},\ i\neq j_{k},\ k=2,...,N \\
p_{j_{k},k} &=&M_{k-1},\ x_{j_{k},k}=\sum_{i=1}^{n}q_{i}x_{i,k-1},\ \text{%
when}\ M_{k-1}=\frac{p_{j_{k},k-1}}{q_{j_{k-1}}}, \\
M_{k} &=&\underset{1\leq i\leq n}{Max}\left( \frac{p_{i,k}}{q_{i}}\right) =%
\frac{p_{j_{k},k}}{q_{j_{k}}},\qquad k=1,...,N,
\end{eqnarray*}%
where $j_{k}$ is a specific index for which $M_{k}$ holds.

With the notations and conditions above the following is obtained:

\begin{theorem}
\bigskip \label{Th3} Let $f:\left[ a,b\right) \rightarrow 
\mathbb{R}
,$ $a\leq b\leq \infty ,$ be a convex function, and let the notations and
conditions above hold. Then, for every integer $N$%
\begin{equation*}
J_{n}\left( f,\mathbf{x}_{1},\mathbf{p}_{1}\right)
-\sum_{k=1}^{N}M_{k}J_{n}\left( f,\mathbf{x}_{k},\mathbf{q}\right) \leq 0,
\end{equation*}%
and 
\begin{equation*}
M_{k}=\frac{p_{j_{1},1}}{q_{j_{1}}^{k}},\qquad k=1,...,N
\end{equation*}%
hold, where $j_{1}$ is a fixed specific integer for which $M_{1}=\underset{%
1\leq i\leq n}{\max }\left( \frac{p_{i,1}}{q_{i}}\right) =\frac{p_{j_{1},1}}{%
q_{j_{1}}}$ is satisfied.
\end{theorem}

In Section 2, theorems \ref{Th6} and \ref{Th7} we extend Theorem \ref{Th2}
and Theorem \ref{Th3}, hence\ also \cite[Theorem 1]{S} by replacing $\mathbf{%
q=}\left( q_{1},...,q_{n}\right) $\ in the $N$ Jensen functionals with $%
\mathbf{q}_{k},$ $k=1,...,N$.

The extensions in theorems \ref{Th8} and \ref{Th9} are obtained by replacing
in Theorem \ref{Th2} and Theorem \ref{Th3} $\mathbf{q=}\left(
q_{1},...,q_{n}\right) $\ in the $N$ Jensen functional with $\mathbf{q}%
_{k}=\left( q_{1,k},...,q_{n_{k},k}\right) $, when $n_{k+1}\leq n_{k}\leq
n_{1},$ $k=1,...,N$. These theorems extend the following theorem in \cite%
{SDB} and \cite{TTD}.

\begin{theorem}
\bigskip \label{Th4} \cite{SDB}. Let $\alpha =\left( \alpha _{1},...\alpha
_{n}\right) ,$ $\beta =\left( \beta ,...,\beta _{n}\right) ,$ $\gamma
=\left( \gamma _{1},...,\gamma _{n}\right) \in p_{n}$ satisfy that $\beta
_{i}+\gamma _{i}>0$\ for each $i=1,...,n$. If $f$\ is a convex function on
an interval $I:=\left[ a,b\right] $\ and $\mathbf{x}=\left(
x_{1},...,x_{n}\right) \in I^{n}$, then 
\begin{eqnarray*}
&&\underset{1\leq i\leq n}{\min }\left\{ \frac{\alpha _{i}}{\beta
_{i}+\gamma _{i}}\right\} \left[ J_{n}\left( f,\mathbf{x},\beta \right)
+J_{n}\left( f,\mathbf{x},\gamma \right) \right] \\
&\leq &J_{n}\left( f,\mathbf{x},\alpha \right) \leq 2\underset{1\leq i\leq n}%
{\max }\left\{ \frac{\alpha _{i}}{\beta _{i}+\gamma _{i}}\right\}
J_{n}\left( f,\mathbf{x},\left( \frac{\beta +\gamma }{2}\right) \right) .
\end{eqnarray*}
\end{theorem}

\begin{theorem}
\label{Th5} \cite[Theorem 2.1]{TTD} Under the hypotheses and notations as in
Theorem \ref{Th1}, we have the Jensen-Dragomir type inequalties 
\begin{equation*}
mJ_{n}\left( f,\mathbf{x},\beta \right) +\mathbf{m}^{\ast }\left( \left\vert
J\right\vert +1\right) H_{j}\leq J_{n}\left( f,\mathbf{x},\alpha \right)
\leq mJ_{n}\left( f,\mathbf{x},\beta \right) +\mathbf{M}^{\ast }\left(
\left\vert J\right\vert +1\right) H_{j},
\end{equation*}%
where \ $J=\left\{ i:\alpha _{i}-\beta _{i}\neq 0\right\} ,$ $\left\vert
J\right\vert $ is the cardinal of $J$, $\mathbf{m}^{\ast }=$\ $\underset{%
i\in J}{min}\left\{ m,\alpha _{i}-m\beta _{i}\right\} ,$ $\mathbf{M}^{\ast }=%
\underset{i\in J}{\max }\left\{ m,\alpha _{i}-m\beta _{i}\right\} $, and%
\begin{equation*}
H_{J}:=\frac{1}{\left\vert J\right\vert +1}\left[ \sum_{i\in J}f\left(
x_{i}\right) +f\left( \sum_{i=1}^{n}\beta _{i}x_{i}\right) \right] -f\left( 
\frac{1}{\left\vert J\right\vert +1}\left( \sum_{i\in
J}x_{i}+\sum_{i=1}^{n}\beta _{i}x_{i}\right) \right) .
\end{equation*}
\end{theorem}

\section{\protect\bigskip Convexity and extended normalized Jensen functional%
}

We start this section with theorems \ref{Th6} and \ref{Th7} which extend
theorems \ref{Th2} and \ref{Th3}. All we need to do in order to prove
theorems \ref{Th6} and \ref{Th7} is to replace in in theoems \ref{Th2} and %
\ref{Th3} $\mathbf{q=}\left( q_{1},...,q_{n}\right) $ with $\mathbf{q}_{k}%
\mathbf{=}\left( q_{1,k},...,q_{n,k}\right) ,$\ $k=1,...,N$. Therefore we do
not find it necessary to prove these two theorems but only to emphasize
where the notations, definitions and conditions differ.\ 

In Theorem \ref{Th6} we use the following conditions and notations:

$0\leq p_{i,1}\leq 1$, $0<q_{i,k}\leq 1,\ i=1,...,n$, $\
\sum_{i=1}^{n}p_{i,1}=\sum_{i=1}^{n}q_{i,k}=1,$ $k=1,...,N$.

Denote $m_{1}=\underset{1\leq i\leq n}{\min }\left( \frac{p_{i,1}}{q_{i,1}}%
\right) $ and $s_{1}$ the number of $i$-th for which $m_{1}$ occur.

Define recoursively%
\begin{eqnarray}
p_{i,k} &=&\left\{ 
\begin{array}{cc}
p_{i,k-1}-m_{k-1}q_{i,k-1}, & m_{k-1}\neq \frac{p_{i,k-1}}{q_{i,k-1}} \\ 
\frac{1}{s_{k-1}}m_{k-1}, & m_{k-1}=\frac{p_{i,k-1}}{q_{i,k-1}}%
\end{array}%
\right. ,\qquad k=2,...  \label{2.1} \\
m_{k-1} &=&%
\begin{array}{cc}
\underset{1\leq i\leq n}{\min }\left( \frac{p_{i,k-1}}{q_{i,k-1}}\right) , & 
k=2,...,%
\end{array}%
\qquad \qquad  \notag
\end{eqnarray}

\bigskip\ and denote $s_{k-1}$ as the number of cases for which $m_{k-1}$
occurs.

Let also $x_{i,1}\in \left( a,b\right) ,$ $i=1,...,n$ and define
recoursively 
\begin{equation}
x_{i,k}={\Large \{}%
\begin{array}{cc}
x_{i,k-1\hspace{0in}}, & m_{k-1}\neq \frac{p_{i,k-1}}{q_{i,k-1}} \\ 
\sum_{i=1}^{n}q_{i,k-1}x_{i,k-1}, & m_{k-1}=\frac{p_{i,k-1}}{q_{i,k-1}}%
\end{array}%
,  \label{2.2}
\end{equation}%
$i=1,...n$, $k=2,...,N$.

We emphasize that $x_{i,k\hspace{0in}},$\ $m_{k}$ in (\ref{2.2}) means that $%
\mathbf{x}_{k},\ $and $m_{k}$, $k=2,...,N$\ are recoursively constructed,
but $\mathbf{q}_{k},$ $k=1,...,N$\ are apriori given and not constructed
recoursively. \ 

\begin{theorem}
\label{Th6} Suppose that $f:\left[ a,b\right) \rightarrow 
\mathbb{R}
,$ $a<b\leq \infty $ is a convex function. Then, for every integer $N,$ 
\begin{equation}
J_{n}\left( f,\mathbf{x}_{1},\mathbf{p}_{1}\right)
-\sum_{k=1}^{N}m_{k}J_{n}\left( f,\mathbf{x}_{k},\mathbf{q}_{k}\right) \geq
0,  \label{2.3}
\end{equation}%
where $\mathbf{p}_{1}=\left( p_{1,1},...,p_{n,1}\right) ,$\ $\mathbf{q}_{k}%
\mathbf{=}\left( q_{1,k},...,q_{n,k}\right) \mathbf{,}$ $\mathbf{x}_{k}=$\ $%
\left( x_{1,k},...,x_{n,k}\right) ,$ $k=1,...,N$, , $\sum_{i=1}^{n}p_{i,1}=%
\sum_{i=1}^{n}q_{i,k}=1,$ $k=1,...,N$, $p_{i,1}\geq 0,$\ $q_{i,k}>0,$ $%
i=1,...,n$, and $m_{k}=\underset{1\leq i\leq n}{\min }\left( \frac{p_{i,k}}{%
q_{i,k}}\right) $ $k=1,...,N$,$\ $where $p_{i,k}$, $m_{k},$ $x_{i,k},$ are
as denoted in (\ref{2.1}) and (\ref{2.2}).
\end{theorem}

Replacing in Theorem \ref{Th6} $\mathbf{q}_{i}=\mathbf{q}$, $i=1,...,N$, we
get Theorem \ref{Th2}. Replacing $\mathbf{q}_{i}=\mathbf{q}$, $i=1,...,N$
and replacing $q_{i},$\ $i=1,...,n$\ by $\frac{1}{n}$\ in Theorem \ref{Th6}
we get \cite[Theorem 1]{S}.

\bigskip

Similarly to the extenssions in Theorem \ref{Th6} of Theorem \ref{Th2} (\cite%
[Theorem 5]{A1}) we extend now in Theorem \ref{Th7} (\cite[Theorem 6]{A1}),
Theorem \ref{Th3} and the left hand-side inequality in Theorem \ref{Th1}.

We denote 
\begin{eqnarray}
&&  \label{2.4} \\
\mathbf{p}_{1} &=&\left( p_{1,1},...,p_{n,1}\right) ,\quad \mathbf{q}_{k}%
\mathbf{=}\left( q_{1,k},...,q_{n,k}\right) ,\quad k=1,...,N  \notag \\
\text{where \ }p_{i,1} &\geq &0,\ q_{i,k}>0,\ i=1,...,n,\
\sum_{i=1}^{n}p_{i,1}=\sum_{i=1}^{n}q_{i,k}=1,\ k=1,...,N  \notag \\
M_{1} &=&\underset{1\leq i\leq n}{Max}\left( \frac{p_{i,1}}{q_{i,1}}\right) =%
\frac{p_{j_{1},1}}{q_{j_{1},1}},\quad  \notag
\end{eqnarray}%
where $j_{1}$ is a fixed specific integer for which $M_{1}$\ holds.

We define recoursively $p_{k}=\left( p_{1,k},...,p_{n,k}\right) ,$ $\mathbf{x%
}_{k}=\left( x_{1,k},...,x_{n,k}\right) ,$ and $M_{k}$, $\ k=1,...,N$ as
follows: 
\begin{eqnarray}
&&  \label{2.5} \\
p_{i,k} &=&p_{i,k-1}-M_{k-1}q_{i,k},\ x_{i,k}=x_{i,k-1},\ \text{when}\
M_{k-1}\neq \frac{p_{i,k-1}}{q_{i,k-1}},\ k=2,...,N  \notag \\
p_{i,k} &=&p_{i,k-1}-M_{k-1}q_{i,k-1},\ x_{i,k}=x_{i,k-1},\ \text{when}\
M_{k-1}=\frac{p_{i,k-1}}{q_{i,k-1}},\ i\neq j_{k},\ k=2,...,N  \notag \\
p_{j_{k},k} &=&M_{k-1},\ x_{j_{k},k}=\sum_{i=1}^{n}q_{i,k-1}x_{i,k-1},\ 
\text{when}\ M_{k-1}=\frac{p_{j_{k-1},k-1}}{q_{j_{k-1},k-1}},  \notag \\
M_{k} &=&\underset{1\leq i\leq n}{Max}\left( \frac{p_{i,k}}{q_{i,k}}\right) =%
\frac{p_{j_{k},k}}{q_{j_{k},k}},\qquad k=1,...,N,  \notag
\end{eqnarray}%
where $j_{k}$\ are specific indices for which $M_{k},$ $k=1,...,N$\ hold.

With the notations and conditions in (\ref{2.4}) and (\ref{2.5}) we get:

\begin{theorem}
\bigskip \label{Th7} Let $f:\left[ a,b\right] \rightarrow 
\mathbb{R}
,$ be a convex function, and let (\ref{2.4}) and (\ref{2.5}) hold. Then, for
every integer $N$%
\begin{equation}
J_{n}\left( f,\mathbf{x}_{1},\mathbf{p}_{1}\right)
-\sum_{k=1}^{N}M_{k}J_{n}\left( f,\mathbf{x}_{k},\mathbf{q}_{k}\right) \leq
0,  \label{2.6}
\end{equation}%
and 
\begin{equation}
M_{k}=\frac{p_{j_{1},1}}{\Pi _{m=1}^{k}\left( q_{j_{1},m}\right) },\qquad
k=1,...,N  \label{2.7}
\end{equation}%
hold, where $j_{1}$ is a fixed specific integer for which $M_{1}=\underset{%
1\leq i\leq n}{Max}\left( \frac{p_{i,1}}{q_{i,1}}\right) =\frac{p_{j_{1},1}}{%
q_{j_{1},1}}$ is satisfied.
\end{theorem}

Theorems \ref{Th8} and \ref{Th9} extend theorems\ref{Th1}, \ref{Th4} and \ref%
{Th5}.

The difference between Theorem \ref{Th6} and Theorem \ref{Th8} is that
inTheorem \ref{Th8} we eliminate all the terms that satisfy $\left(
p_{j_{i},k}-m_{k}q_{j_{i},k}\right) f\left( x_{i}\right) $ where $%
p_{j_{i},k}-m_{k}q_{j_{i},k}=0$\ and add instead only one term equal to $%
m_{k}f\left( \sum_{i=1}^{n_{k}}q_{i,k}x_{i,k}\right) $ therefore $\mathbf{q}%
_{k}=\left( q_{1},...,q_{n_{k}}\right) ,$ $k=1,...,N$\ satisfy $n_{1}\geq
n_{2}\geq ...\geq n_{N}$.\ 

In Theorem \ref{Th8} we use the following conditions and notations:

$0\leq p_{i,1}\leq 1$, $0<q_{i,1}\leq 1$, $i=1,...,n_{1},\
\sum_{i=1}^{n_{1}}p_{i,1}=\sum_{i=1}^{n_{1}}q_{i,1}=1,$ $m_{1}=\underset{%
i=1,...,n_{1}}{\min }\left( \frac{p_{i,1}}{q_{i,1}}\right) $.

We define recoursively,%
\begin{eqnarray}
&&%
\begin{array}{cccc}
p_{i,k}= & p_{j_{i},k-1}-m_{k-1}q_{j_{i},k-1}, & m_{k-1}\neq \frac{p_{i,k-1}%
}{q_{i,k-1}} & i=1,...,n_{k-1}-1 \\ 
p_{n_{k},k}= & m_{k-1}, & m_{k-1}=\frac{p_{i,k-1}}{q_{i,k-1}} & 
\end{array}
\label{2.8} \\
&&%
\begin{array}{ccc}
q_{i,k-1}=q_{j_{i},k-1}, & i=1,...,n_{k-1}, & k=2,...,N%
\end{array}%
\qquad  \notag
\end{eqnarray}%
\begin{equation*}
m_{k-1}=%
\begin{array}{cc}
\underset{1\leq i\leq n_{k-1}}{\min }\left( \frac{p_{i,k-1}}{q_{i,k-1}}%
\right) , & k=2,...N,%
\end{array}%
\qquad \qquad \qquad \qquad \qquad \qquad
\end{equation*}

and%
\begin{equation}
\begin{array}{cccc}
x_{i,k}= & x_{j_{i},k-1\hspace{0in}}, & m_{k-1}\neq \frac{p_{i,k-1}}{%
q_{i,k-1}} & i=1,...,n_{k-1}-1 \\ 
x_{n_{k},k}= & \sum_{i=1}^{n}q_{i,k-1}x_{i,k-1}, & m_{k-1}=\frac{p_{i,k-1}}{%
q_{i,k-1}} & 
\end{array}
\label{2.9}
\end{equation}%
$k=2,...,N$.

Because of the similar techniques to those in the proof of Theorem \ref{Th2}
(see \cite{A1}), we show here only an outline of the proof.

\begin{theorem}
\label{Th8} Suppose that $f:\left[ a,b\right] \rightarrow 
\mathbb{R}
,$ is a convex function. Then, for every integer $N,$ 
\begin{equation*}
J_{n_{1}}\left( f,\mathbf{x}_{1},\mathbf{p}_{1}\right)
-\sum_{k=1}^{N}m_{k}J_{n_{k}}\left( f,\mathbf{x}_{k},\mathbf{q}_{k}\right)
\geq 0,
\end{equation*}%
holds where $\mathbf{p}_{1}=\left( p_{1,1},...,p_{n_{1},1}\right) ,$\ $%
\mathbf{q}_{k}\mathbf{=}\left( q_{1,k},...,q_{n_{k},k}\right) \mathbf{,}$ $%
\mathbf{x}_{k}=$\ $\left( x_{1,k},...,x_{n_{k},k}\right) ,$ $n_{1}\geq
n_{2}\geq ,...,\geq n_{k}\geq n_{k+1}\geq n_{N},\ k=1,...,N$, and $p_{i,k}$, 
$m_{k},$ $x_{i,k},$ are as denoted in (\ref{2.8}) and (\ref{2.9}), $%
\sum_{i=1}^{n_{1}}p_{i,1}=\sum_{i=1}^{n_{k}}q_{i,k}=1,$ $k=1,...,N$\ and $%
p_{i,1}\geq 0,\ \ i=1,...,n_{1}$, \ $q_{i,k}>0,\ i=1,...,n_{k}$, $m_{k}=%
\underset{1\leq i\leq n_{k}}{\min }\left( \frac{p_{i,k}}{q_{i,k}}\right) ,$ $%
k=1,...,N.$
\end{theorem}

\begin{proof}
\bigskip According to (\ref{2.8}) and (\ref{2.9})%
\begin{eqnarray}
&&\sum_{i=1}^{n_{1}}p_{i,1}f\left( x_{i,1}\right) -m_{1}\left(
\sum_{i=1}^{n_{1}}q_{i,1}f\left( x_{i,1}\right) -f\left(
\sum_{i=1}^{n_{1}}q_{i,1}x_{i,1}\right) \right)  \label{2.10} \\
&=&\sum_{i=1}^{n_{1}}\left( p_{i,1}-m_{1}q_{i,1}\right) f\left(
x_{i,1}\right) +m_{1}f\left( \sum_{i=1}^{n_{1}}q_{i,1}x_{i,1}\right)  \notag
\\
&=&\sum_{i=1}^{n_{2}-1}\left( p_{i,1}-m_{1}q_{i,1}\right) f\left(
x_{i,1}\right) +m_{1}f\left( \sum_{i=1}^{n_{1}}q_{i,1}x_{i,1}\right) , 
\notag \\
&=&\sum_{i=1}^{n_{2}}p_{i,2}f\left( x_{i,2}\right) ,  \notag
\end{eqnarray}%
The expression $\sum_{i=1}^{n_{2}-1}\left( p_{i,1}-m_{1}q_{i,1}\right)
f\left( x_{i,1}\right) $ includes only the terms $\left(
p_{i,1}-m_{1}q_{i,1}\right) >0$.

Therefore, from (\ref{2.10}) we get that \ 
\begin{eqnarray*}
&&\sum_{i=1}^{n_{1}}p_{i,1}f\left( x_{i,1}\right) -m_{1}\left(
\sum_{i=1}^{n_{1}}q_{i,1}f\left( x_{i,1}\right) -f\left(
\sum_{i=1}^{n_{1}}q_{i,1}x_{i,1}\right) \right) \\
&=&\sum_{i=1}^{n_{2}}p_{i,2}f\left( x_{i,2}\right) ,
\end{eqnarray*}%
and it is easy to verify that 
\begin{equation}
\sum_{i=1}^{n_{2}}p_{i,2}=1,\quad p_{i,2}\geq 0,\quad
\sum_{i=1}^{n_{2}}p_{i,2}x_{i,2}=\sum_{i=1}^{n_{1}}p_{i,1}x_{i,1}
\label{2.11}
\end{equation}%
\ \ \ \ \ \ \ Hence, 
\begin{eqnarray*}
&&\sum_{i=1}^{n_{1}}p_{i,1}f\left( x_{i,1}\right) -\sum_{k=1}^{N}m_{k}\left(
\sum_{i=1}^{n_{k}}q_{i,k}f\left( x_{i,k}\right) -f\left(
\sum_{i=1}^{n_{k}}q_{i,k}x_{i,k}\right) \right) \\
&=&\sum_{i=1}^{n_{2}}p_{i,2}f\left( x_{i,2}\right)
-\sum_{k=2}^{N}m_{k}\left( \sum_{i=1}^{n_{k}}q_{i,k}f\left( x_{i,k}\right)
-f\left( \sum_{i=1}^{n_{k}}q_{i,k}x_{i,k}\right) \right) \\
&=&\qquad \qquad \cdot \qquad \qquad \cdot \qquad \qquad \cdot \qquad \qquad
\qquad \qquad \cdot \\
&=&\sum_{i=1}^{n_{N}}p_{i,N}f\left( x_{i,N}\right) -m_{N}\left(
\sum_{i=1}^{n_{N}}q_{i,N}f\left( x_{i,N}\right) -f\left(
\sum_{i=1}^{n_{N}}q_{i,N}x_{i,N}\right) \right) \\
&=&\sum_{i=1}^{n_{N}}\left( p_{i,N}-m_{N}q_{i,N}\right) f\left(
x_{i,N}\right) +m_{N}f\left( \sum_{i=1}^{n_{N}}q_{i,N}x_{i,N}\right) \\
&=&\sum_{i=1}^{n_{N+1}-1}\left( p_{i,N}-m_{N}q_{i,N}\right) f\left(
x_{i,N}\right) +m_{N}f\left( \sum_{i=1}^{n_{N}}q_{i,N}x_{i,N}\right) \\
&=&\sum_{i=1}^{n_{N+1}}p_{i,N+1}f\left( x_{i,N+1}\right) .
\end{eqnarray*}

Also, it is clear, the same way as in (\ref{2.11}) that%
\begin{eqnarray}
\sum_{i=1}^{n_{k}}p_{i,k} &=&1,\quad p_{i,k}\geq 0,\quad i=1,...,n_{k},\quad
k=1,...,N+1,  \label{2.12} \\
\sum_{i=1}^{n_{1}}p_{i,1}x_{i,1} &=&\sum_{i=1}^{n_{k}}p_{i,k}x_{i,k},\quad
k=1,...,N+1.  \notag
\end{eqnarray}%
Therefore, as a result we have to show in order to prove the theorem that%
\begin{eqnarray}
&&\sum_{i=1}^{n_{1}}p_{i,1}f\left( x_{i,1}\right) -\sum_{k=1}^{N}m_{k}\left(
\sum_{i=1}^{n_{k}}q_{i,k}f\left( x_{i,k}\right) -f\left(
\sum_{i=1}^{n_{k}}q_{i,k}x_{i,k}\right) \right)  \label{2.13} \\
&=&\sum_{i=1}^{n_{N+1}}p_{i,N+1}f\left( x_{i,N+1}\right) \geq f\left(
\sum_{i=1}^{n_{1}}p_{i,1}x_{i,1}\right) =f\left(
\sum_{i=1}^{n_{N+1}}p_{i,N+1}x_{i,N+1}\right) ,  \notag
\end{eqnarray}%
holds, and inequality (\ref{2.13}) is satisfied because of the convexity of $%
f$, (\ref{2.12}) and $\sum_{i=1}^{n_{N+1}}p_{i,N+1}\left( x_{i,N+1}\right)
=\sum_{i=1}^{n_{1}}p_{i,1}x_{i,1}$.

The proof of the theorem is complete.
\end{proof}

In Theorem \ref{Th9}, similarly to Theorem \ref{Th8}, where we extend
Theorem \ref{Th2} and the right handside of the inequality in Theorem \ref%
{Th1}, we extend now Theorem \ref{Th3} and the left hand-side inequality in
Theorem \ref{Th1}.

The difference between Theorem \ref{Th7} and Theorem \ref{Th9} is that in
Theorem \ref{Th9} we eliminate all the terms that satisfy $%
p_{j_{i},k}-M_{k}q_{j_{i},k}=0$\ and add instead only one term equal to $%
m_{k}f\left( \sum_{i=1}^{n_{k}}q_{i,k}x_{i,k}\right) $. Therefore, $\mathbf{q%
}_{k}=\left( q_{1,k},...,q_{n_{k},k}\right) ,$ $\mathbf{x}_{k}=\left(
x_{1,k},...,x_{n_{k},k}\right) ,$ $k=1,...,N$\ satisfy $n_{1}\geq n_{2}\geq
,...,\geq n_{N}$.\ Moreover, we prove that $n_{k}=n_{2},\ k=2,...,N.$

We denote 
\begin{eqnarray}
\mathbf{p}_{1} &=&\left( p_{1,1},...,p_{n_{1},1}\right) ,\quad \mathbf{q}_{1}%
\mathbf{=}\left( q_{1,1},...,q_{n_{1},1}\right) ,  \label{2.14} \\
\text{where\quad }p_{i,1} &\geq &0,\quad q_{i,1}>0,\quad i=1,...,n_{1},\quad
\sum_{i=1}^{n_{1}}p_{i,1}=\sum_{i=1}^{n_{1}}q_{i,1}=1,  \notag \\
M_{1} &=&\underset{1\leq i\leq n_{1}}{Max}\left( \frac{p_{i,1}}{q_{i,1}}%
\right) =\frac{p_{j_{1},1}}{q_{j_{1},1}},  \notag
\end{eqnarray}%
for a fixed specific integer $j_{1}$ which satisfies $M_{1}$.

We also define recoursively 
\begin{eqnarray}
&&\quad  \label{2.15} \\
\mathbf{x}_{k} &=&\left( x_{1,k},...,x_{n_{k},k}\right) ,\quad k=1,...,N 
\notag \\
p_{i,k} &=&p_{i,k-1}-M_{k-1}q_{i,k-1},\ x_{i,k}=x_{i,k-1},\ \text{when}\
M_{k-1}\neq \frac{p_{i,k-1}}{q_{i,k-1}},\ k=2,...,N  \notag \\
p_{j_{k},k} &=&M_{k-1},\
x_{j_{k},k}=\sum_{i=1}^{n_{k-1}}q_{i,k-1}x_{i,k-1},\ \text{when}\ M_{k-1}=%
\frac{p_{j_{k-1},k-1}}{q_{j_{k-1},k-1}},  \notag \\
M_{k} &=&\underset{1\leq i\leq n_{k}}{Max}\left( \frac{p_{i,k}}{q_{i,k}}%
\right) =\frac{p_{j_{k},k}}{q_{j_{k},k}},\qquad k=1,...,N,  \notag
\end{eqnarray}%
where each $j_{k},$is a specific index for which $M_{k}$ holds, $k=1,...,N$.

As in Theorem \ref{Th8}, because of the similar techniques to those in \cite%
{A1} we show in Theorem \ref{Th9} only an outline of the proof.

With the notations, definitions and conditions in (\ref{2.14}) and (\ref%
{2.15}) we get:

\begin{theorem}
\bigskip \label{Th9} Let $f:\left[ a,b\right) \rightarrow 
\mathbb{R}
,$ $a\leq b<\infty ,$ be a convex function, and let (\ref{2.14}) and (\ref%
{2.15}) hold. Then, for every integer $N$%
\begin{equation}
J_{n_{1}}\left( f,\mathbf{x}_{1,}\mathbf{p}_{1}\right)
-\sum_{k=1}^{N}M_{k}J_{k}\left( f,\mathbf{x}_{k},\mathbf{q}_{k}\right) \leq
0,  \label{2.16}
\end{equation}

\begin{equation}
M_{k}=\frac{p_{j_{1},1}}{\Pi _{m=1}^{k}\left( q_{j_{1},m}\right) },Ck=1,...,N
\label{2.17}
\end{equation}%
and 
\begin{equation}
n_{1}\geq n_{2}=n_{k},=n_{N},\quad k=2,...,N  \label{2.18}
\end{equation}%
hold, where $j_{1}$ is a fixed specific integer for which $M_{1}=\underset{%
1\leq i\leq n_{1}}{\max }\left( \frac{p_{i,1}}{q_{i,1}}\right) =\frac{%
p_{j_{1},1}}{q_{j_{1},1}}$ is satisfied.
\end{theorem}

\begin{proof}
As $j_{1}$ is a specific integer for which $M_{1}=\underset{1\leq i\leq n_{1}%
}{Max}\left( \frac{p_{i,1}}{q_{i},_{1}}\right) =\frac{p_{j_{1},1}}{%
q_{j_{1},1}},$\ it is easy to see that%
\begin{equation*}
M_{k-1}=\underset{1\leq i\leq n_{k-1}}{Max}\frac{p_{i,k-1}}{q_{i,k-1}}=\frac{%
p_{j_{1},1}}{\Pi _{m=1}^{k-1}q_{j_{1}},m},\quad k=2,...,N+1
\end{equation*}%
because the only positive $p_{i,k},$ $k=2,...$ is $p_{j_{1},k}$, as 
\begin{equation*}
\text{when }\frac{p_{i,1}}{q_{i,1}}<\frac{p_{j_{1},1}}{q_{j_{1},1}},\quad
i\neq j_{1},\quad \text{then }p_{i,2}<0,\quad x_{i,2}=x_{i,1},
\end{equation*}

\ 
\begin{equation*}
\text{when }i=j_{1},\quad \text{then }p_{i,2}\ >0,\quad
\end{equation*}%
and \ 
\begin{equation*}
\sum_{i=1}^{n_{2}}p_{i,2}x_{i,2}=1,\quad
x_{j_{1},2}=\sum_{i=1}^{n_{1}}q_{i,1}x_{i,1.}.
\end{equation*}%
Hence, also for $k=2,...,N$\ the only positive $p_{i,k},$\ $i=2,...,n_{k}$
is $p_{j_{1},k}$\ where $j_{1}$ is the fixed integer that satisfies $M_{1}=%
\underset{1\leq i\leq n_{1}}{\max }\left( \frac{p_{i,1}}{q_{i,1}}\right) =%
\frac{p_{j_{1},1}}{q_{j_{1},1}},$ and therefore we can replace in (\ref{2.15}%
) $j_{k}$ with $j_{1}$, which means that (\ref{2.17}) holds. Also, as the
only positive $p_{i,k}$, $k=2,...,N$\ is $p_{j_{k},k}$,\ it is obvious that $%
n_{k}=n_{2}$, $k=2,...,N$.

In order to complete the proof of the theorem, we proceed now with\ proving (%
\ref{2.16}):

In a similar way as in Theorem \ref{Th7}\ together with (\ref{2.18}) we get

\begin{eqnarray}
&&  \label{2.19} \\
&&\sum_{i=1}^{n_{1}}p_{i,1}f\left( x_{i,1}\right) -\sum_{k=1}^{N}M_{k}\left(
\sum_{i=1}^{n_{k}}q_{i,k}f\left( x_{i,k}\right) -f\left(
\sum_{i=1}^{n_{k}}q_{i,k}x_{i,k}\right) \right)  \notag \\
&=&\sum_{i=1}^{n_{2}}p_{i,2}f\left( x_{i,2}\right)
-\sum_{k=2}^{N}M_{k}\left( \sum_{i=1}^{n_{k}=n_{2}}q_{i,k}f\left(
x_{i,k}\right) -f\left( \sum_{i=1}^{n_{k}=n_{2}}q_{i,k}x_{i,k}\right) \right)
\notag \\
&=&\qquad \cdot \qquad \qquad \qquad \cdot \qquad \qquad \cdot \qquad \qquad
\qquad \qquad \cdot  \notag \\
&=&\sum_{i=1}^{n_{N}=n_{2}}p_{i,N}f\left( x_{i,N}\right) -M_{N}\left(
\sum_{i=1}^{n_{N}=n_{2}}q_{i,N}f\left( x_{i,N}\right) -f\left(
\sum_{i=1}^{n_{N}=n_{2}}q_{i,N}x_{i,N}\right) \right)  \notag \\
&=&\sum_{i=1}^{n_{N}=n_{2}}\left( p_{i,N}-M_{N}q_{i}\right) f\left(
x_{i,N}\right) +M_{N}f\left( \sum_{i=1}^{n_{N}=n_{2}}q_{i,N}x_{i,N}\right) 
\notag \\
&=&\sum_{i=1}^{n_{N+1}=n_{2}}p_{i,N+1}f\left( x_{i,N+1}\right)  \notag
\end{eqnarray}

Also it is clear that

\begin{eqnarray*}
p_{i,k} &\leq &0,\quad i=1,...,n_{k},\quad i\neq j_{1},\quad
p_{j_{1},k}>0,\quad k=2,...,N+1, \\
\sum_{i=1}^{n_{k+1}=n_{2}}p_{i,k+1} &=&1,\quad
\sum_{i=1}^{n_{1}}p_{i,1}x_{i,1}=\sum_{i=1}^{n_{k}=n_{2}}p_{i,k}x_{i,k},%
\quad k=2,...,N+1.
\end{eqnarray*}

From (\ref{2.19}) it follows that in order to prove (\ref{2.16}) we have to
show that%
\begin{eqnarray*}
&&\sum_{i=1}^{n_{1}}p_{i,1}f\left( x_{i,1}\right) -M_{1}\left(
\sum_{i=1}^{n_{1}}q_{i,1}f\left( x_{i,1}\right) -f\left(
\sum_{i=1}^{n_{1}}q_{i,1}x_{i,1}\right) \right) \\
&&-\sum_{k=2}^{N}M_{k}\left( \sum_{i=1}^{n_{k}=n_{2}}q_{i,k}f\left(
x_{i,k}\right) -f\left( \sum_{i=1}^{n_{k}=n_{2}}q_{i,k}x_{i,k}\right) \right)
\\
&=&\sum_{i=1}^{n_{N+1}=n_{2}}p_{i,N+1}f\left( x_{i,N+1}\right) \leq f\left(
\sum_{i=1}^{n_{1}}p_{i,1}x_{i,1}\right) ,
\end{eqnarray*}%
that is, we have to show that%
\begin{equation*}
\sum_{i=1}^{n_{N}=n_{2}}\left( p_{i,N}-M_{N}q_{i,N}\right) f\left(
x_{i,N}\right) +M_{N}f\left( \sum_{i=1}^{n_{N}=n_{2}}q_{i,N}x_{i,N}\right)
\leq f\left( \sum_{i=1}^{n_{1}}p_{i,1}x_{i,1}\right) .
\end{equation*}%
In other words, we have to show that 
\begin{eqnarray*}
&&\frac{1}{M_{N}}f\left( \sum_{i=1}^{n_{N}=n_{2}}p_{i,N}x_{i,N}\right)
+\sum_{i=1,\ i\neq j_{1}}^{n_{N}=n_{2}}\left( q_{i,N}-\frac{p_{i,N}}{M_{N}}%
\right) f\left( x_{i,N}\right) \\
&\geq &f\left( \sum_{i=1}^{n_{N+1}=n_{2}}q_{i,N}x_{i,N.}\right)
\end{eqnarray*}%
holds.

The last inequality follows from the convexity of $f$ because $q_{i,N}-\frac{%
p_{i,N}}{M_{N}}\geq 0\ ,$ $i=1,...,n_{N},$ $i\neq j_{1}$\ \ and $\frac{1}{%
M_{N}}>0$,\ therefore, the inequality (\ref{2.16}) holds.

The proof of the theorem is complete.
\end{proof}

\begin{remark}
\label{Rem1} If in Theorem \ref{Th9} $M_{1}=\underset{1\leq i\leq n_{1}}{Max}%
\left( \frac{p_{i,1}}{q_{i,1}}\right) =\frac{p_{j_{1},1}}{q_{j_{1},1}}>\frac{%
p_{i,1}}{q_{i,1}},$ $i=1,...,n,$ $i\neq j_{1}$, then Theorem \ref{Th9} is
the same as Theorem \ref{Th7}.
\end{remark}

\textbf{Concluding Remarks.} \textit{In Theorem \ref{Th6} we eliminate all }$%
s_{k-1}$\textit{\ terms for which }$\left( p_{i,k-1}-m_{k-1}q_{i,k-1}\right)
f\left( x_{i,k-1}\right) =0$\textit{\ and replace all of these }$s_{k-1}$%
\textit{\ with terms }$\frac{m_{k-1}}{s_{k-1}}f\left(
\sum_{i=1}^{n}q_{i,k-1}x_{i,k-1}\right) $\textit{. On the other hand, in
Theorem \ref{Th8} we eliminate all the }$s_{k-1}$\textit{\ terms for which }$%
\left( p_{i,k-1}-m_{k-1}q_{i,k-1}\right) =0$\textit{\ and replace only one
of them with }$m_{k-1}f\left( \sum_{i=1}^{n}q_{i,k-1}x_{i,k-1}\right) $%
\textit{.}

\textit{From theorems \ref{Th6} and \ref{Th8} it is easy to realize that we
can prove similarly to these two theorems a theorem in which we eliminate
all }$s_{k-1}$\textit{\ terms for which }$\left(
p_{i,k-1}-m_{k-1}q_{i,k-1}\right) =0$\textit{\ and replace }$r_{k-1}$\textit{%
\ terms }$1\leq r_{k-1}\leq s_{k-1}$\textit{\ and get similar results. The
same can be done as a consequence of what is done in theorems \ref{Th7} and %
\ref{Th9}.}

\bigskip

\end{document}